\documentclass[a4paper,final]{amsart}

\usepackage{amsfonts}
\usepackage{amssymb}
\usepackage{enumerate}
\usepackage{amsmath}
\usepackage[notcite,notref]{showkeys}
\usepackage{amssymb}

\newtheorem{theorem}{Theorem}[section]
\newtheorem{lemma}[theorem]{Lemma}
\newtheorem{prop}[theorem]{Proposition}

\newtheorem{question}[theorem]{Question}
\newtheorem{cor}[theorem]{Corollary}

\theoremstyle{definition}                   
\newtheorem{remark}[theorem]{Remark}

\newtheorem{defi}[theorem]{Definition}
\newtheorem{notation}[theorem]{Notation}

\newcommand{\cont}{\mathfrak{c}}

\newcommand{\N}{\mathbb{N}}

\newcommand{\Q}{\mathbb{Q}}
\newcommand{\R}{\mathbb{R}}

\newcommand{\om}{\omega}
\newcommand{\si}{\sigma}
 
\newcommand{\de}{\delta}

\def\beq{\begin{equation}}
\def\eeq{\end{equation}}

\newcommand{\bt}{\begin{theorem}}
\newcommand{\bl}{\begin{lemma}}
\newcommand{\br}{\begin{remark}}
\newcommand{\bd}{\begin{defi}}

\newcommand{\et}{\end{theorem}}
\newcommand{\el}{\end{lemma}}
\newcommand{\er}{\end{remark}}
\newcommand{\ed}{\end{defi}}
\newcommand{\bp}{\begin{proof}}
\newcommand{\ep}{\end{proof}}

\DeclareMathOperator{\cov}{cov}

\newcommand{\su}{\subset}

\newcommand{\RR}{\mathbb{R}}

\newcommand{\iG}{\mathcal{G}}
\newcommand{\iH}{\mathcal{H}}
\newcommand{\iI}{\mathcal{I}}

\newcommand{\iK}{\mathcal{K}}

\newcommand{\iM}{\mathcal{M}}

\newcommand{\iN}{\mathcal{N}}

\newcommand{\sm}{\setminus}
\newcommand{\la}{\lambda}
\newcommand{\ka}{\kappa}
\newcommand{\al}{\alpha}
\newcommand{\be}{\beta}
\renewcommand{\phi}{\varphi}

\title[Decomposing the real line]
{Decomposing the real line into Borel sets closed under addition}

\author{M\'arton Elekes}
\address{Alfr\'ed R\'enyi Institute of Mathematics\\
PO Box 127, 1364 Budapest, Hungary\\
and
Institute of Mathematics\\
E\"otv\"os Lor\'and University\\
P\'azm\'any P\'eter s.~1/c, 1117
Budapest, Hungary}
\email{elekes.marton@renyi.mta.hu}
\urladdr{http://www.renyi.hu/~emarci}
\thanks{We gratefully acknowledge the support of the
Hungarian Scientific Foundation grants no.~104178 and 83726.
}

\author{Tam\'as Keleti}
\address{Alfr\'ed R\'enyi Institute of Mathematics\\
PO Box 127, 1364 Budapest, Hungary\\
and
Institute of Mathematics\\
E\"otv\"os Lor\'and University\\
P\'azm\'any P\'eter s.~1/c, 1117
Budapest, Hungary}
\email{tamas.keleti@gmail.com}
\urladdr{http://www.cs.elte.hu/analysis/keleti}

\begin{document}
\begin{abstract}
We consider decompositions of the real line into pairwise disjoint Borel pieces so that each piece is closed under addition. How many pieces can there be? We prove among others that the number of pieces is either at most 3 or uncountable, and we show that it is undecidable in $ZFC$ and even in the theory $ZFC + c=\om_2$ if the number of pieces can be uncountable but less than the continuum. We also investigate various versions: what happens if we drop the Borelness requirement, if we replace addition by multiplication, if the pieces are subgroups, if we partition $(0,\infty)$, and so on.
\end{abstract}



\maketitle

\section{Introduction}

We consider decompositions of the real line into Borel subsets with some additional algebraic structure. In this paper the words `decomposition' and `partition' will always refer to writing the real line as a union of pairwise disjoint sets. The main question is weather the number of pieces can be strictly between $\om$ and $\cont$. First, the question is interesting if we impose no algebraic structure at all. 
Lebesgue was the first one to show that the real line can be decomposed into $\om_1$ Borel sets, hence it is consistent with $ZFC$ that the number of pieces can be strictly between $\om$ and $\cont$. 
Later Hausdorff \cite{Ha} showed that there is in fact a partition into $\om_1$ many $F_{\si\de}$ sets. Whether the real line can be partitioned into $\om_1$ many $G_{\de\si}$ sets is already independent from $ZFC$, and even from $ZFC + c = \om_2$, see \cite{Mi} for the details.
Finally, J. Stern \cite{Sn} and independently A. W. Miller \cite{Mi} proved that consistently $\cont=\om_2$ and 
the real line can be partitioned into $\om_1$ many compact sets.

It is important to mention the following  remarkable theorem of Silver: If an equivalence relation on the line (considered as a subset of the plane) is Borel then there are countably many or continuum many equivalence classes \cite{Si}. This shows that in all these decompositions into $\ka \in (\om, \cont)$ many Borel pieces the relation of being in the same piece is fairly complicated. We will return to this issue in the Open questions section.

In this paper the algebraic assumption about the pieces will mostly be that each piece is closed under addition. In other words, we partition the real line into Borel additive semigroups. Since $(-\infty,0),\{0\}$ and $(0,\infty)$ are closed under addition, we can trivially decompose $\R$ into one, two or three Borel sets that are closed under addition. We show (Theorem~\ref{t:addstructure}) that in any other decomposition of the real line into Borel additive semigroups each semigroup must have Lebesgue measure zero and must be of first category. Therefore the number of pieces is either $1,2,3$ or uncountable (in fact, at least $\max(\cov\iM,\cov\iN)$). The sets of the form $c\cdot\Q^+$ clearly decompose $\R$ into continuum many sets that are closed under addition. So it is consistent with $ZFC$ that the set of the possible
number of pieces is $\{1,2,3,\cont\}$. Can this be proved in $ZFC$? 

Our main result (Theorem~\ref{t:main}) is that it is consistent with $ZFC$ that $\cont=\omega_2$ and $\R$ can be decomposed into $\om_1$ Borel sets that are closed under addition.

Now, in the remaining part of the introduction we say a few words about certain natural variations of the problem.

If we do not require the sets to be Borel then the problem is much easier and the answer is much simpler. Throughout the paper $\ka$ and $\la$ will denote cardinal numbers.

\begin{prop}
For any $1\le\ka\le\cont$ there exists a decomposition 
$\R=\bigcup_{\al<\ka} A_{\al}$ such that each $A_{\al}$ is nonempty and
closed under addition.
\end{prop}

\bp
Let $H=\{h_\al\ : \al<\cont\}$ be a Hamel basis of $\R$.
For every $x\in\R\sm\{0\}$ take its unique representation in the form
$x=r_1 h_{\al_1} + \ldots +r_n h_{\al_n}$, where $r_1,\ldots,r_n\in\Q\sm\{0\}$
and $\al_1<\ldots<\al_n<\cont$ and
let $f(x)=\al_1$ and $g(x)=r_1$.
For any $\al<\cont$ let
$$
B_\al=\{x\ : f(x)=\al, g(x)>0\}.
$$
Clearly every $B_\al$ is nonempty and closed under addition.
The same is true for $C_\be=\R\sm(\cup_{\al<\be} B_\al)$ for any $\be\le\cont$
since $x\in C_\be \Leftrightarrow
       f(x)\ge \be \textrm{ or } g(x)<0 \textrm{ or } x=0$.

Let $\ka'$ be $\ka-1$ if $\ka$ is finite and $\ka$ otherwise.
Then $\R=C_{\ka'} \cup \bigcup_{\al<\ka'} B_\al$ is a decomposition
with the required properties.
\ep

If we consider multiplication instead of addition then we get similar results (both for the Borel and the non-Borel case), but the negative numbers cause some extra technical complications, see Theorem~\ref{t:multstructure} and Corollary~\ref{c:multkappa}. Hence it is perhaps more natural to decompose only $(0,\infty)$ into subsets that are closed under multiplication. Taking logarithm it is clear that such a decomposition is 
equivalent to a decomposition of $\R$ into sets that are closed under addition. Recently R. Freud \cite{F} raised the question if $(0,\infty)$ can be decomposed into two parts that are closed under \emph{both} addition and multiplication. It turned out that this had already been solved in 2007 by D. M. Kane, who had proved that such a decomposition exists. 
Another answer was given in \cite{KST}, where the authors also describe the structure of all such decompositions. They also show that for any $1\le\ka\le \cont$ one can decompose $(0,\infty)$ into $\ka$ sets that are closed under addition and multiplication.

So far we considered additive or multiplicative semigroups. It is natural to ask what happens if we require groups. Then of course we cannot hope to get decompositions since all groups contain the neutral element. So we should rather require that their intersection only contains the neutral element. In this case we say that the subgroups are \emph{essentially disjoint}.
But it turns out that even in that case the answer is fairly easy, even in a more general setting. Recall that $|X|$ denotes the cardinality of a set $X$.

\begin{prop}
If an infinite group $G$ is the union of $\ka$ many essentially disjoint 
subgroups  then $\ka=1$ or $\ka = |G|$.
\end{prop}

\begin{proof}
Let $G = \bigcup_{\al<\ka} H_{\al}$ be a union consisting of essentially disjoint nonempty subgroups. Assume $\ka \ge 2$. Since the $H_{\al}$'s are essentially disjoint, $\ka$ cannot exceed $|G|$, so we only need to show that $\ka \ge |G|$. 

Let $\al < \ka$ and $g  \notin H_\al$. Then $|gH_\al \cap H_\beta| \le 1$ for every $\beta \neq \al$, since if $g_1, g_2 \in gH_\al \cap H_\beta$ then $g_1^{-1}g_2 \in H_\al \cap H_\beta = \{e\}$, hence $g_1 = g_2$. Since $gH_\al \cap H_\al = \emptyset$, the $H_\beta$'s cover $gH_\al$, hence $\ka \ge |H_\al|$. Repeating the argument for every $\al$ yields $\ka \ge \sup\{|H_\al| : \al < \ka\}$. Let $\la = \sup\{|H_\al| : \al < \ka\}$, then $\ka \ge \la$. If $\la = |G|$ then $\ka\ge\la=|G|$ and we are done, so let $\la < |G|$. Since the $H_\al$'s cover $G$, we obtain $\ka \la \ge |G|$, 
in particular, $\ka$ is infinite. Then $\ka \la = \ka$, therefore $\ka \ge |G|$, hence the proof is complete.
\end{proof}

We remark here that it is easy to check that if $\R$ is 
the union of essentially disjoint subgroups, then these groups are actually $\Q$-linear vector spaces. 
Hence there is a close connection to the theory of so called vector space partitions, which deals with the problem of how one can 
write a vector space as the union of essentially disjoint proper subspaces.
However, this theory mostly considers finite vector spaces, since one of the main motivations is the connection to error-correcting codes. See e.g. \cite{Be}, \cite{Bu}, \cite{EZ}, \cite{He} and the references therein for more details.

\section{Non-existence results concerning Borel decompositions}
\label{s:zfc}

In this section we prove structural results about decompositions into additive and/or multiplicative Borel semigroups, which in turn yield strong limitations about the possible number of pieces. All results of the section are proved in $ZFC$.

\begin{notation}
For any $A\su\R$, $n\in\N^+$ let $(n)A=\{a_1+\ldots+a_n :  a_n\in A\}$.
\end{notation}

The following lemma is clearly well-known, however, for the sake of completeness we include its proof.

\begin{lemma}\label{l:evensum}
Let $B\su\R$ be a Borel set and suppose that $B$ has positive Lebesgue 
measure or $B$ is of second category. 


(i) Then $\bigcup_{k=1}^\infty (2k)B$ contains a halfline.

(ii) If we also have $B+B\su B$ or $B+B+B\su B$ then $B$ contains a halfline.
\end{lemma}

\begin{proof}
(i)
If $B$ has positive Lebesgue measure
then Steinhaus theorem \cite{St}, if $B$ is of second category then
Piccard theorem \cite{Pi} (see also in \cite{Ke}) 
implies that $B+B$ contains an interval $(a,b)$, so 
$(2k)B \supset (ka,kb)$ for any $k\in\N^+$.
Since for large enough $k$ the consecutive intervals $(ka,kb)$ overlap, 
and so $\cup_{k=1}^\infty(ka,kb)$ contains a halfline, 
this completes the proof of (i).

(ii)
If $B+B\su B$ then $B$ contains $\bigcup_{k=1}^\infty (2k)B$
so we are done by (i).

Now suppose that $B+B+B\su B$. Then, by induction we get that $(2k+1)B\su B$
for any $k\in\N^+$. Thus we have $\bigcup_{k=1}^\infty (2k+1)B\su B$.
Since $\bigcup_{k=1}^\infty (2k+1)B = B + \bigcup_{k=1}^\infty (2k)B$ and 
$B$ is nonempty, (i) implies that 
$\bigcup_{k=1}^\infty (2k+1)B$ contains a halfline,
which completes the proof of (ii).
\end{proof}

\begin{lemma}\label{l:pospartition}
Suppose that 
we have a decomposition 
$(0,\infty)=\bigcup_{\al<\ka} B_{\al}$ such that every $B_\al$
is Borel and for each $\al<\ka$ we have 
$B_\al+B_\al\su B_\al$ or $B_\al+B_\al+B_\al \su B_\al$.
Then either every $B_\al$ is of Lebesgue measure zero and
is of first category or $B_{\al_1}=(0,\infty)$ for some $\al_1<\ka$.
\end{lemma}

\begin{proof}
Suppose that there is an $\al_1<\ka$
such that $B_{\al_1}$ has positive Lebesgue measure or is of second
category. We will show that $B_{\al_1}=(0,\infty)$.

By Lemma~\ref{l:evensum}, 
$B_{\al_1}$ contains a halfline, so the other sets of the decomposition
are all bounded. But clearly no bounded nonempty set $B\su(0,\infty)$
can have the property $B+B\su B$ or $B+B+B\su B$, so we have 
$B_{\al_1}=(0,\infty)$.
\end{proof}

\begin{theorem}\label{t:addstructure}
Suppose that 
we have a decomposition 
$\R=\bigcup_{\al<\ka} A_{\al}$ such that each $A_{\al}$ is
a Borel set closed under addition.
Then every $A_\al$ is of Lebesgue measure zero and
is of first category, or is equal to
$\R$, $(-\infty,0)$, $(-\infty,0]$, $(0,\infty)$, or $[0,\infty)$.
\end{theorem}

\begin{proof}
If $\R=\bigcup_{\al<\ka} A_{\al}$ is a decomposition such that each $A_{\al}$ is
closed under addition then
$(0,\infty)=\bigcup_{\al<\ka} A_{\al}\cap (0,\infty)$ is also a decomposition
and each $A_{\al}\cap (0,\infty)$ is closed under addition. 
Then by Lemma~\ref{l:pospartition}, 
$A_{\al}\cap (0,\infty)$ is either $(0,\infty)$ or a set of Lebesgue
measure zero and of first category.
By symmetry, we also have that $A_{\al}\cap (-\infty,0)$ is 
either $(-\infty,0)$ or a set of Lebesgue
measure zero and of first category. 

It remains to prove that if 
$A_{\al}$ contains $(0,\infty)$ or $(-\infty,0)$ then 
$A_\al$ equals $\R$, $(-\infty,0)$, $(-\infty,0]$, $(0,\infty)$ or
$[0,\infty)$. 
So suppose that $A_\al\supset (0,\infty)$ and $a\in A_\al\cap(-\infty,0)$.
Then, since $A_\al$ is closed under addition, we get $A_\al\supset(a,\infty)$,
so $A_\al\cap (-\infty,0)$ has positive measure, thus by the previous
paragraph $A_{\al}\supset (-\infty,0)$, therefore $A_{\al}=\R$.
Similarly, we can prove that if 
$A_\al\supset (-\infty,0)$ and $a\in A_\al\cap(0,\infty)$ then $A_{\al}=\R$,
which completes the proof.
\end{proof}

Recall that if $\iI$ is an ideal on a set $X$ then $\cov\iI = \min \{|\iH| : \iH \su \iI, \cup \iH = X \}$. Let $\iN$ and $\iM$ denote the ideal of Lebesgue nullsets and the ideal of sets of first category in $\R$, respectively.
 
\begin{cor}\label{c:addkappa}
Suppose that 
we have a decomposition 
$\R=\bigcup_{\al<\ka} A_{\al}$ such that each $A_{\al}$ is
a nonempty Borel set closed under addition.
Then $\ka\in\{1,2,3\}$ or $\ka\ge \max(\cov\iN,\cov\iM)$.
\end{cor}

\begin{proof}
If $\ka<\max(\cov\iN,\cov\iM)$ then at most $\ka$ sets of Lebesgue
measure zero and of first category cannot cover $(-\infty,0)$ or
$(0,\infty)$. Then
by Theorem~\ref{t:addstructure}, every $A_\al$ equals to 
$\{0\}$, $\R$, $(-\infty,0)$, $(-\infty,0]$, $(0,\infty)$, or $[0,\infty)$,
so $\ka\le 3$.
\end{proof}

\begin{theorem}\label{t:multstructure}
Suppose that 
we have a decomposition 
$\R=\bigcup_{\al<\ka} M_{\al}$ such that each $M_{\al}$ is
a Borel set closed under  \emph{multiplication}.
Then every $M_\al$ is of Lebesgue measure zero and
is of first category, or
is the union 
of some of the sets $(-1,0)\cup (0,1)$, 
$(-\infty,-1)\cup (1,\infty)$, $\{-1,1\}$ and  $\{0\}$.
%
\end{theorem}

\begin{proof}
Let $B_{\al}=M_\al\cap (-\infty,-1)$. Then $B_\al$ is not closed
under multiplication but we still have $B_\al \cdot B_\al \cdot B_\al \su B_\al$
since $M_\al \cdot M_\al \cdot M_\al \su M_\al$ and 
$(-\infty,-1) \cdot (-\infty,-1) \cdot (-\infty,-1) \su (-\infty,-1)$.
Let $C_{\al}=\log(-B_\al)$. Then $\bigcup_{\al<\ka} C_{\al}$ 
is a Borel decomposition of 
$(0,\infty)$ and we have $C_\al+C_\al+C_\al\su C_\al$. 
Then by Lemma~\ref{l:pospartition}, either every $C_\al$ is of Lebesgue
measure zero and is of first category or
there exists an $\al_1<\ka$ for which $C_{\al_1}=(0,\infty)$.
In the first case every $B_{\al}=M_\al\cap (-\infty,-1)$ is also 
of Lebesgue
measure zero and is of first category.
In the latter case, 
for this $\al_1$, we have $(-\infty,-1)=B_{\al_1}=M_{\al_1}\cap (-\infty,-1)$.
Since $M_{\al_1}$ is closed under multiplication, it also contains 
$(1,\infty)$, so we have $M_{\al_1}\supset (-\infty,-1)\cup (1,\infty)$.

Repeating the above argument for $B'_\al=M_\al\cap (-1,0)$ and 
$C'_\al= -\log(-B'_\al)$ we get that either
every $B'_\al=M_\al\cap (-1,0)$ is of Lebesgue measure zero and of first
category or there exists an $\al_2<\ka$ such that 
$M_{\al_2}\supset (-1,0)\cup (0,1)$.

Since the set $M_\al$ that contains $-1$ also contains $1$,
the only fact that remains to prove is that if $M_\al$ contains one of
$(-\infty,-1)\cup (1,\infty)$ and $(-1,0)\cup (0,1)$ then either it
contains the other one or it is disjoint from the other one.  
So suppose that $M_\al\supset (-\infty,-1)\cup (1,\infty)$ and 
$a\in M_\al\cap((-1,0)\cup (0,1))$.
Then, since $M_\al$ is closed under multiplication, we get 
$M_\al\supset(-\infty,-a)\cup (a,\infty)$,
so $M_\al \cap((-1,0)\cup (0,1))$ has positive measure, thus by 
the above 
paragraph $M_{\al}\supset (-1,0)\cup (0,1)$.
Similarly, we can prove that if 
$M_\al\supset (-1,0)\cup (0,1)$ and 
$a\in M_\al\cap (-\infty,-1)\cup (1,\infty))$ 
then $M_{\al}\supset(-\infty,-1)\cup (1,\infty)$,
which completes the proof.
\end{proof}

\begin{remark}
Similarly to Theorem~\ref{t:addstructure}, we could also explicitly describe 
in Theorem~\ref{t:multstructure} the possible options for
$M_\al$, in case it is not a Lebesgue measure zero set of first category.
We cannot get all the $2^4=16$ possible unions of  the sets $(-1,0)\cup (0,1)$, 
$(-\infty,-1)\cup (1,\infty)$, $\{-1,1\}$ and  $\{0\}$, since some of these
unions are not closed under multiplication. We do not have to take those ones
that are of measure zero and first category. But it is easy to see that $M_\al$ can be any of the
remaining 10, namely $(-1,0)\cup (0,1)$, $[-1,0)\cup (0,1]$,
$(-1,1)$, $[-1,1]$,
$(-\infty,-1)\cup (1,\infty)$, $(-\infty,-1]\cup [1,\infty)$,
$(-\infty,-1)\cup \{0\} \cup (1,\infty)$, 
$(-\infty,-1]\cup \{0\} \cup [1,\infty)$,
$\R$ and $\R\sm\{0\}$.
\end{remark}

\begin{cor}\label{c:multkappa}
Suppose that 
we have a decomposition 
$\R=\bigcup_{\al<\ka} M_{\al}$ such that each $M_{\al}$ is
a nonempty Borel set closed under multiplication.
Then $\ka\in\{1,2,3,4\}$ or $\ka\ge \max(\cov\iN,\cov\iM)$.
\end{cor}

\begin{proof}
If $\ka<\max(\cov\iN,\cov\iM)$ then at most $\ka$ sets of Lebesgue
measure zero and of first category cannot cover any interval. Then
by Theorem~\ref{t:multstructure} every $M_\al$ 
is the union 
of some of the sets $(-1,0)\cup (0,1)$, 
$(-\infty,-1)\cup (1,\infty)$, $\{-1,1\}$ and  $\{0\}$, so $\ka\le 4$. 
\end{proof}

\begin{cor}\label{c:combinedstructure}
Suppose that 
we have a decomposition 
$\R=\bigcup_{\al<\ka} B_{\al}$ such that each $B_{\al}$ is
a nonempty Borel set closed under \emph{both} addition and multiplication.

(i) Then either every $B_\al$ is of Lebesgue measure zero and
is of first category or $B_{\al_1}=\R$ for some $\al_1<\ka$.

(ii) We have $\ka=1$ or $\ka\ge \max(\cov\iN,\cov\iM)$

\end{cor}

\begin{proof}
Claim (i) follows directly from Theorems~\ref{t:addstructure} 
and \ref{t:multstructure}. Claim (ii) follows from (i).
\end{proof}

\section{The main result: Existence of a certain Borel decomposition}
\label{s:main}


\begin{theorem}\label{t:main}
It is 
consistent that $\cont=\om_2$ and 
the real line can be partitioned into $\om_1$ additive $F_{\si}$ semigroups;
that is, there exists a decomposition 
$\R=\bigcup_{\al<\om_1} A_{\al}$ such that each $A_{\al}$ is a
nonempty $F_{\si}$ set
closed under addition. 
\end{theorem}

\begin{proof}
Ciesielski and Pawlikowski \cite{CP} proved that
consistently $\cont=\om_2$ and there exist a Hamel basis 
that is the union of $\om_1$ pairwise disjoint Cantor sets.
Fix such a Hamel basis $H$ with such a decomposition 
$H=\bigcup_{\al<\om_1} C_{\al}$, and in each Cantor set $C_{\al}$ fix 
a countable base 
$\{B_{\al}^k\}_{k \in \N^+}$ such that for each $k$ and $\al$ both
$B_{\al}^k$ and $C_\al\sm B_{\al}^k$ are compact.


Every $x\in\R$ has a unique representation of the form 
$x=r_1 b_1 + \ldots +r_n b_n$, where $r_1,\ldots,r_n\in\Q\sm\{0\}$ and 
$b_1,\ldots,b_n\in H$. For notational simplicity, let us define 
$B(x)=\{b_1,\ldots,b_j\}$, and for any $J\su H$
we say that $r_i$ is a \emph{$J$-coefficient of $x$} if $b_i\in J$.
For $x\in\R$ let 
$$
I(x)=\{\ \al<\om_1\ :\ x \textrm{ has a } C_{\al} \textrm{-coefficient }\} 
= \{\ \al<\om_1\ :\ B(x)\cap C_{\al}\neq \emptyset\ \}.
$$ 
For $x\in\R$ and $J\su H$ also let $S(x,J)$ be the sum of the $J$-coefficients of $x$.
Note that $S(x+y,J) = S(x,J) + S(y,J)$ for every $J\su H$ and $x, y\in\R$.

Finally, let 
\begin{equation}\nonumber
\begin{split}
A^+(\al,k) & =\{x \in \R\sm\{0\} : \max(I(x))=\al,\ S(x,B_{\al}^k)>0,\ 
                                      (\forall i<k)\ S(x,B_{\al}^i)=0   \},\\ 
A^-(\al,k) & =\{x \in \R\sm\{0\} : \max(I(x))=\al,\ S(x,B_{\al}^k)<0,\ 
                                      (\forall i<k)\ S(x,B_{\al}^i)=0   \}. 
\end{split}
\end{equation}

In other words, $A^+(\al,k)$ contains $x\in\R\sm\{0\}$ if and only if 
$\al$ is the largest ordinal so that $x$ has a $C_{\al}$-coefficient
and 
$k$ is the smallest integer so that the sum
of the $B_{\al}^k$-coefficients of $x$ is nonzero and this sum is
positive, and similarly for $A^-(\al,k)$.

We claim that $\{0\}$ and the 
sets $A^+(\al,k)$ and $A^-(\al,k)$ for $\al<\om_1$ and 
$k\in\N^+$ form a decomposition with all the required properties.

It is clear from the definitions that these sets are pairwise disjoint. 
To show that their union is $\R$ it is enough to check that if 
$x$ has a $C_{\al}$-coefficient then $S(x,B_{\al}^i)$ is nonzero for
some $i\in\N^+$. So suppose that $x$ has a $C_{\al}$-coefficient, that is, 
$B(x)\cap C_{\al}$ is a nonempty finite subset of $C_{\al}$.
Since $\{ B_{\al}^i \}_{i \in \N^+}$ is a base of $C_{\al}$,
there exists an $i$ such that $B(x)\cap B_{\al}^i$ 
is a singleton, hence $S(x,B_{\al}^i)$, as the sum of a single nonzero term,
cannot be zero.

Now we show that the 
sets $A^+(\al,k)$ and $A^-(\al,k)$ are closed under addition.
By symmetry, it is enough to prove this for $A^+(\al,k)$.
If $x,y\in A^+(\al,k)$ then clearly 
$S(x+y,B_{\al}^k)=S(x,B_{\al}^k)+S(y,B_{\al}^k)>0$
and $S(x+y,B_{\al}^i)=S(x,B_{\al}^i)+S(y,B_{\al}^i)=0$ for every $i<k$.
Since $S(x+y,B_{\al}^k)>0$, $x+y$ has a $C_\al$-coefficient. On
the other hand, $x$ and $y$ have no $C_\be$-coefficients for $\be>\al$,
so $x+y$ cannot have $C_\be$-coefficients either, thus 
$\max(I(x+y))=\al$. Therefore $x+y\in A^+(\al,k)$, indeed.

So it remains to prove that the sets  $A^+(\al,k)$ and $A^-(\al,k)$ are
$F_{\sigma}$, and again it is enough to prove this for $A^+(\al,k)$.
Fix $\al<\om_1$ and $k\in\N^+$. Note that $A^+(\al,k)=D_\al+E^+(\al,k)$, where
$$
D_{\al}=\big\{x\in\R\sm\{0\}\ :\ B(x)\su \cup_{\be<\al} C_\be\big\},
$$
$$
E^+(\al,k)=\big\{x\in\R\sm\{0\}\ :\ B(x)\su C_\al,\ S(x,B_{\al}^k)>0,\ 
                                      (\forall i<k)\ S(x,B_{\al}^i)=0 \big\},
$$
and we use the notation $A+B=\{a+b:a\in A, b\in B\}$. The sum of two 
compact sets is compact, hence the sum of two $\si$-compact sets is 
$\si$-compact, thus it is enough to check that both $D_\al$ and 
$E^+(\al,k)$ are $\sigma$-compact.

The set $D_\al$ is $\si$-compact since
$$
D_{\al}=\bigcup\ \{\ r_1 C_{\be_1}+\ldots+{r_m C_{\be_m}} \ :\ m\in\N^+, 
(\forall i=1,\ldots m)\ r_i\in\Q,\ \be_i<\al\ \}.
$$

In order to show that $E^+(\al,k)$ is $\sigma$-compact, we rewrite it as
$$
E^+(\al,k)=F^+(\al,k)\cap F^0(\al,1) \cap F^0(\al,2) \cap \ldots 
\cap F^0(\al,k-1),
$$
where for any $n\in\N^+$,
\begin{equation}\nonumber
\begin{split}
F^0(\al,n) &=\{\ x \in \R\sm\{0\} \ :\ B(x)\su C_{\al},\ S(x,B_{\al}^n)=0\  \},\\ 
F^+(\al,n) &=\{\ x \in \R\sm\{0\} \ :\ B(x)\su C_{\al},\ S(x,B_{\al}^n)>0\  \}.
\end{split}
\end{equation}
Thus, to complete to proof, it is enough to show that the sets $F^0(\al,n)$
and $F^+(\al,n)$ are $\si$-compact.
Note that these sets can be also written 
as
\begin{equation}\nonumber
\begin{split}
F^0(\al,n) = \bigcup\ \Big\{\ & r_1 B_{\al}^n+\ldots+{r_m B_{\al}^n} 
+ r_{m+1}(C_{\al}\sm B_{\al}^n)+\ldots+r_{m+l}(C_{\al}\sm B_{\al}^n) \ :\\
 & m,l\in\N,\ r_1,\ldots,r_{m+l}\in\Q,\ r_1+\ldots+r_m=0\ \Big\},\\ 
F^+(\al,n) = \bigcup\ \Big\{\ & r_1 B_{\al}^n+\ldots+{r_m B_{\al}^n} 
+ r_{m+1}(C_{\al}\sm B_{\al}^n)+\ldots+r_{m+l}(C_{\al}\sm B_{\al}^n) \ :\\
 & m,l\in\N,\ r_1,\ldots,r_{m+l}\in\Q,\ r_1+\ldots+r_m>0\ \Big\}. 
\end{split}
\end{equation}
Since the sets  $B_{\al}^n$ and $C_{\al}\sm B_{\al}^n$ were chosen to be
compact, these sets are $\si$-compact indeed, which completes the proof.
\end{proof}

The following simple observations show that Theorem~\ref{t:main} is sharp in 
the sense that $F_\si$ cannot be replaced by $G_\de$.

\begin{prop}
Suppose that for some cardinal $\ka$ we have a decomposition 
$\R=\bigcup_{\al<\ka} A_{\al}$ such that each $A_{\al}$ is
closed under addition and nonempty. 

(i) Then each $A_{\al}$ is the union of sets of the form 
$c\cdot \Q^+$ ($c\in\R$).

(ii) If each $A_{\al}$ is a nonempty $G_{\de}$ set then each $A_\al$ is the union 
of some of the sets $(-\infty,0), \{0\}, (0,\infty)$, consequently $\ka\le 3$. 
\end{prop}

\begin{proof}
(i) Suppose that $x,y\in\R$, $x/y\in\Q^+$, $x\in A_\al$ and $y\in A_\be$.
Since $x/y\in\Q^+$, $x$ and $y$ has a common positive integer multiple $z$.
But then $z$ is both in $A_{\al}$ and $A_{\be}$, so we have
$\al=\be$, which completes the proof of (i). 

(ii) If each $A_{\al}$ is a nonempty $G_\de$ set then, by (i), each $A_\al$ is 
residual in $(-\infty,0)$ or in $(0,\infty)$ or equals to $\{0\}$,
which completes the proof of (ii).
\end{proof}

Combining the results of this section and the previous one, in case of
$\cont=\omega_2$ we can exactly determine how many Borel additive semigroup 
one can decompose $\R$ into:

\begin{cor}
Let $\iK$ be the set of cardinalities of all possible decompositions of
$\R$ into Borel additive semigroup; that is, let 
$$
\iK=\{\ka: \exists \textrm{ Borel decomposition }
\cup_{\al<\ka} B_{\al} =\R \textrm{ with } 
(\forall\al) B_\al+B_\al\su B_\al, B_\al\neq \emptyset
\}.
$$
If $\cont=\omega_2$ then $\iK=\{1,2,3,\om_1,\om_2\}$ or 
$\iK=\{1,2,3, \om_2\}$ and both possibilities are consistent.
\end{cor}

\begin{proof}
Suppose that $\cont=\omega_2$.
Since $\max(\cov\iM, \cov\iN)\ge\om_1$, Corollary~\ref{c:addkappa}
implies that $\iK\su\{1,2,3,\om_1,\om_2\}$. The trivial decompositions 
$\R=\R$, $\R=(-\infty,0)\cup[0,\infty)$, 
$\R=(-\infty,0)\cup\{0\}\cup(0,\infty)$ show that 
$\iK\supset\{1,2,3\}$.
Let $\{J_\al:\al<\cont\}$ be the equivalence classes of the relation
$x\sim y \Leftrightarrow (x=y=0 \textrm{ or } x/y\in\Q)$. 
Then each $J_\al$ is countable and closed
under addition, so the decomposition $\R=\bigcup_{\al<\cont} J_\al$ shows that
$\om_2=\cont\in\iK$.

Thus it remains to prove that both $\omega_1\in\iK$ and $\omega_1\not\in\iK$
are consistent if $\cont=\omega_2$. Theorem~\ref{t:main} shows the consistency
of $\omega_1\in\iK$. It is well known (see \cite{BJ}) that it is consistent
that $\cov\iM=\cov\iN=\om_2=\cont$. But in this case, 
by Corollary~\ref{c:addkappa}, $\omega_1\not\in\iK$, which completes the proof.
\end{proof}

\section{Open questions}

In this final section we collect some of the numerous remaining open questions.

\begin{question}
Is it consistent that $\cont = \om_3$ and
the real line can be partitioned into $\om_2$ additive Borel (or $F_\si$) semigroups;
that is, there exists a decomposition 
$\R=\bigcup_{\al<\om_1} A_{\al}$ such that each $A_{\al}$ is a
nonempty Borel (or $F_{\si}$) set
closed under addition?
\end{question}

We remark here that it is not hard to see from \cite{Sn} or \cite{Mi} that it is consistent that $\cont = \om_3$ and
the real line can be partitioned into $\om_2$ compact sets.

\begin{question}
What can we say about the following cardinal invariant?
$$
\gamma=\min\{\ka: \ka>3, \exists \textrm{ Borel dec. }
\cup_{\al<\ka} B_{\al} =\R \textrm{ with }
(\forall\al) B_\al+B_\al\su B_\al, B_\al\neq \emptyset\}.
$$
\end{question}

By Corollary~\ref{c:addkappa} we know that $\gamma\ge\max(\cov\iN,\cov\iM)$
and by Theorem~\ref{t:main} we know that it is consistent
that $\cont=\om_2$ and $\gamma=\om_1$. What else can we say?

\smallskip

As we pointed out above, by Silver's theorem the construction in Theorem \ref{t:main} cannot yield a Borel equivalence relation. (Two real numbers are equivalent if they are in the same piece of the decomposition, and the equivalence relation is Borel if it is Borel when considered as a subset of the plane.) Actually, Silver proved his theorem for so called co-analytic equivalence relations, see \cite{Ke} for the definition and some background in descriptive set theory.
Now we remark that our construction is not analytic either. Indeed, by a result of Stern \cite{Sn2} if each equivalence class of an analytic equivalence relation is $F_\si$ then there are countably many or continuum many classes. However, we do not know the answer to the following.

\begin{question}
Is it consistent that $\cont = \om_2$ and
the real line can be partitioned into $\om_1$ additive Borel semigroups so that the resulting equivalence relation is analytic?
\end{question}

It would also be interesting to check whether our construction consistently produces a projective equivalence relation.

\medskip

The following question is closely related to our topic in that it also 
requires some extra condition using the additive structure of the real
line.

\begin{question}
Suppose that 
$\R = \bigcup_{\al<\kappa} B_\al$ is a decomposition into Borel sets that
are \emph{translates} of each other. Does this imply that $\kappa \le \om$ is
or $\ka = \cont$?
\end{question}

First we show that $\ka\le\om$ or $\ka\ge\cov\iN$, so 
an affirmative answer is consistent with ZFC.
Indeed, if $B_\al$ has measure zero then clearly $\ka\ge\cov\iN$.
If $B_\al$ has positive measure then we have $\ka$ disjoint translates
of a set of positive measure, and it is easy to check that this implies $\ka\le\om$.

We also claim that the answer is affirmative in $ZFC$ if the sets $B_\al$
are $F_\si$. 
Indeed, suppose that $\om < \kappa$. By \cite[Thm. 1]{Ba} a $\si$-compact set has either at most countably many or continuum many pairwise disjoint translates, hence $B_0$ has continuum many pairwise disjoint translates 
$\{B_0+x_\al: \al<\cont\}$. 
One can easily check that this implies that no translate of $B_0$ can contain more
than one number $-x_\al$, 
hence less than $\cont$ many translates of $B_0$ cannot cover $\R$, therefore $\kappa = \cont$. 

This proof also raises the following natural question.

\begin{question}
Suppose that a Borel subset of $\R$ has uncountably many pairwise disjoint translates. Does it also have continuum many pairwise disjoint translates?
\end{question}

The last question is the natural continuation of the question of Freud.

\begin{question}
Suppose that 
$(0, \infty) = \bigcup_{\al<\kappa} B_\al$ is a decomposition into nonempty Borel sets that
are closed under both addition and multiplication. Does this imply that $\kappa = 1$ or $\ka = \cont$?
\end{question}

We remark here that the case $\ka = \cont$ is possible indeed, as was pointed out by Andr\'as M\'ath\'e \cite{Ma}. Let $\iG$ be the smallest family of real functions closed under addition and multiplication and also forming a group under composition (in particular, $\iG$ contains the identity function). Note that this makes sense, since we can generate this family inside the group of strictly increasing functions. Then $\iG$ is clearly countable and it is easy to see that the orbits of $\iG$ form a partition of $\RR$ into $\cont$ many countable sets all of which are closed under addition and multiplication.


\begin{thebibliography}{99}

\bibitem{Ba} T.~Banakh, N.~Lyaskovska, D.~Repov\v{s}, Packing index of subsets in Polish groups. \textit{Notre Dame J. Form. Log.} \textbf{50} (2009), no. 4, 453--468. 

\bibitem{BJ} T.~Bartoszynski, H.~Judah, \textsl{Set Theory: On the Structure of the Real Line.} \textit{A K Peters, Ltd., Wellesley, MA}, 1995.

\bibitem{Be} A.~Beutelspacher, Partitions of finite vector spaces: an application of the Frobenius number in geometry.
\textit{Arch. Math. (Basel)} \textbf{31} (1978/79), no. 2, 202--208. 

\bibitem{Bu} T.~Bu, Partitions of a vector space. \textit{Discrete Math.} \textbf{31} (1980), no. 1, 79--83. 

\bibitem{CP} K.~Ciesielski, J.~Pawlikowski,
Nice Hamel bases under the covering property axiom,
\textsl{Acta Math. Hungar.} \textrm{105} (2004), 197--213.

\bibitem{EZ} S.~I.~El-Zanati, G.~F.~Seelinger, P.~A. Sissokho, L.~E.~Spence, C.~Vanden Eynden, Partitions of finite vector spaces into subspaces. \textit{J. Combin. Des.} \textbf{16} (2008), no. 4, 329--341.

\bibitem{F} R.~Freud, personal communication.

\bibitem{Ha} F.~Hausdorff, Summen von $\aleph_1$ Mengen. \textit{Fund. Math.} \textbf{26} (1936), no. 1, 241--255. 

\bibitem{He} O.~Heden, The Frobenius number and partitions of a finite vector space.
\textit{Arch. Math. (Basel)} \textbf{42} (1984), no. 2, 185--192. 

\bibitem {Ka} D.~M.~Kane, A Partition of the Positive Reals into Algebraically Closed Subsets, 
unpublished, 
http://math.stanford.edu/~dankane/algebraicallycloseddecomposition.pdf, 2007.

\bibitem{Ke} A.~S.~Kechris, \textsl{Classical descriptive set theory.} Graduate Texts in Mathematics, 156. \textit{Springer-Verlag, New York}, 1995.

\bibitem{KST} G.~Kiss, G.~Somlai, T.~Terpai, Decompositions of the positive real numbers into sets closed under addition and multiplication, in preparation.

\bibitem{Ma} A.~M\'ath\'e, personal communication.

\bibitem{Mi} A.~W.~Miller, Covering $2^\om$ with $\om_1$ disjoint closed sets. \textit{The Kleene Symposium (Proc. Sympos., Univ. Wisconsin, Madison, Wis., 1978)}, 415--421, Stud. Logic Foundations Math., \textit{101}, North-Holland, Amsterdam-New York, 1980. 

\bibitem{Pi} S.~Piccard, 
\textsl{Sure les ensembles de distances de ensembles de points d'un espace Euclidean} (M\^emoires de l'Universit\'e de Neuch\^atel, vol. 13), 
L'Universit\'e de Neuch\^atel, 1939.

\bibitem{Si} J.~H.~Silver, Counting the number of equivalence classes of Borel and coanalytic equivalence relations.
\textit{Ann. Math. Logic} \textbf{18} (1980), no. 1, 1--28. 

\bibitem{St} H.~Steinhaus, 
Sur les distances des points des ensembles de mesure positive,
\textsl{Fund. Math.} \textrm{1} (1920), 93--104.

\bibitem{Sn} J.~Stern, Partitions of the real line into $\aleph_1$ closed sets. \textit{Higher set theory (Proc. Conf., Math. Forschungsinst., Oberwolfach, 1977)}, 455-Â460, Lecture Notes in Math., 669, \textit{Springer, Berlin}, 1978. 

\bibitem{Sn2} J.~Stern, Effective partitions of the real line into Borel sets of bounded rank. \textit{Ann. Math. Logic} \textbf{18} (1980), no. 1, 29--60.

\end{thebibliography}
\end{document}